\documentclass[11pt,reqno]{amsart}

\usepackage[T1]{fontenc}
\usepackage[boldsans]{ccfonts}
\usepackage{eulervm}
\renewcommand{\mathbf}{\mathbold}
\usepackage{bbm}

\usepackage{enumerate, amsmath, amsthm, amsfonts, amssymb, 
mathrsfs, graphicx, paralist} 
\usepackage[usenames, dvipsnames]{color}
\usepackage[margin=1in]{geometry} 
\usepackage{hyperref}
\usepackage{comment}

\usepackage{mathabx} 
\usepackage{tikz-cd}

\numberwithin{equation}{section}
\newtheorem{Theorem}[equation]{Theorem}
\newtheorem{Proposition}[equation]{Proposition}

\newtheorem{Corollary}[equation]{Corollary}

\theoremstyle{definition}
\newtheorem{Remark}[equation]{Remark}

\newtheorem{eg}[equation]{Example}

\newcommand{\cF}{\mathcal{F}}

\newcommand{\cG}{\mathcal{G}}

\newcommand{\cN}{\mathcal{N}}

\newcommand{\cO}{\mathcal{O}}

\newcommand{\cS}{\mathcal{S}}

\newcommand{\cT}{\mathcal{T}}

\newcommand{\cV}{\mathcal{V}}

\newcommand{\br}{\mathbf{r}}

\renewcommand{\AA}{\mathbb{A}}

\newcommand{\OO}{\mathbb{O}}
\newcommand{\PP}{\mathbb{P}}

\newcommand{\VV}{\mathbb{V}}

\renewcommand{\phi}{\varphi}

\renewcommand{\tilde}[1]{\widetilde{#1}}

\newcommand{\leftexp}[2]{\vphantom{#2}^{#1} #2}

\makeatletter
\def\Ddots{\mathinner{\mkern1mu\raise\p@
\vbox{\kern7\p@\hbox{.}}\mkern2mu
\raise4\p@\hbox{.}\mkern2mu\raise7\p@\hbox{.}\mkern1mu}}
\makeatother

\DeclareMathOperator{\Spec}{Spec}

\newcommand{\suchthat}{\mid}

\newcommand{\textand}{\text{ }\mathrm{and}\text{ }}

\newcommand{\kk}{\Bbbk}

\DeclareMathOperator{\Mat}{Mat}

\newcommand{\affGr}{\cG r}
\newcommand{\Wedge}{\bigwedge}
\newcommand{\sh}{sh}
\newcommand{\Hom}{\operatorname{Hom}}
\newcommand{\primeST}{\leftexp{'}{\cS^T}}
\newcommand{\GL}{GL}
\newcommand{\SL}{SL}
\newcommand{\Gr}{\mathrm{Gr}}
\newcommand{\SGr}{\mathrm{SGr}}

\begin{document}

\title[On a conjecture of Pappas and Rapoport]{On a conjecture of Pappas and Rapoport \\ about the standard local model for $\GL_d$}
\author{Dinakar Muthiah}
\address{Kavli Institute for the Physics and Mathematics of the Universe (WPI), The University of Tokyo Institutes for Advanced Study, The University of Tokyo, Kashiwa, Chiba 277-8583, Japan}
\email{dinakar.muthiah@ipmu.jp}
\author{Alex Weekes}
\address{Department of Mathematics, University of British Columbia}
\email{weekesal@math.ubc.ca}
\author{Oded Yacobi}
\address{School of Mathematics and Statistics, University of Sydney}
\email{oded.yacobi@sydney.edu.au}
\maketitle

\begin{abstract}
  In their study of local models of Shimura varieties for totally ramified extensions, Pappas and Rapoport posed a conjecture about the reducedness of a certain subscheme of $n \times n$ matrices. We give a positive answer to their conjecture in full generality. Our main ideas follow naturally from two of our previous works. The first is our proof of a conjecture of Kreiman, Lakshmibai, Magyar, and Weyman on the equations defining type A affine Grassmannians. The second is the work of the first two authors and Kamnitzer on affine Grassmannian slices and their reduced scheme structure. We also present a version of our argument that is almost completely elementary: the only non-elementary ingredient is the Frobenius splitting of Schubert varieties.
\end{abstract}

\section{Introduction}

Shimura varieties serve as a bridge between arithmetic geometry and automorphic forms, and as such they play a central role in the Langlands program.
Rapoport and Zink (\cite{Rapoport-Zink}) introduced the study of \emph{local models} of Shimura varieties.  These local models are intended to capture the behaviour that occurs when considering reduction mod $p$ and often allow one to reduce arithmetic problems to questions about affine Grassmannians and flag varieties. These problems tend to be difficult but often tractible (e.g.~\cite{He,Zhu-2014,Haines-Richarz}). 

Local models are defined over the ring of integers $\cO_E$ of a local field $E$ (the reflex field).  An essential desideratum is that the local model be flat over $\cO_E$, which involves a set-theoretic condition called topological flatness and a scheme-theoretic condition about the location of embedded primes (see e.g.~\cite[\S 14.3]{Gortz-Wedhorn}). For example, a topologically-flat family over $\cO_E$ with reduced special fiber is flat.

In \cite{Pappas-Rapoport}, Pappas and Rapoport study the {\em standard (local) model for $\GL_d$} in the case of a totally ramified extension. As Pappas previously observed (\cite{Pappas}), this model is not topologically flat in general. Topological flatness does hold in the cases where, in their notation, the $(r_\phi)_\phi$ differ by at most $1$ (\cite[Proposition 3.2]{Pappas-Rapoport}).  In these cases, flatness of the standard model would follow from the reducedness of the following subscheme of $n \times n$ matrices (\cite[Theorem B (iii), Corollary 5.9]{Pappas-Rapoport}):
\begin{align}
  \label{eq:cN-n-e-space}
   \cN_{n,e} = \left\{ A \in \Mat_{n\times n} \suchthat  \ A^e = 0, \  \ \det( \lambda  - A) = \lambda^n \right\}
\end{align}
Pappas and Rapoport state this as \cite[Conjecture 5.8]{Pappas-Rapoport}. Our main result is a positive answer to their conjecture.

A consequence of our result is that local models in this case represent explicit moduli functors.  Additionally, Pappas and Rapoport prove interesting results about nilpotent orbits and affine Grassmannians conditional upon the reducedness of $\cN_{n,e}$ (\cite[Equation 5.26, Proposition 6.5, Proposition 6.6]{Pappas-Rapoport}). We can now state these results unconditionally.

Questions of flatness of local models have been studied by other authors. G\"ortz proved flatness of unramified local models for the general linear and symplectic groups \cite{Gortz-2001,Gortz-2003}. We mention also results on topological flatness by G\"ortz and by Smithling (\cite{Gortz-2005,Smithling-2011a,Smithling-2011b,Smithling-2014}).

 \subsubsection{Weyman's work}
 Soon after the Pappas-Rapoport conjecture was first announced in 2000, Weyman (\cite{Weyman}) gave proofs of the conjecture (i) when the base field is characteristic zero, and (ii) when the base field has arbitrary characteristic and $e = 2$.  Our main contribution is proving the conjecture in positive characteristic, which is the interesting case for arithmetic applications. We make no assumptions on the characteristic, so our argument also gives a new proof in characteristic zero.

\subsection{How the space $\cN_{n,e}$ appears }

We recall very briefly how $\cN_{n,e}$ appears in the work of Pappas and Rapoport.  We refer the reader to \cite[\S 2,3,4]{Pappas-Rapoport} for the details. This discussion is for motivation, and will not be needed for the rest of the paper. 

Fix a complete discretely valued field $F_0$ with perfect residue field, and let $F/F_0$ be a totally ramified extension of degree $e$. Fix a uniformizer $\pi_0$ of $F_0$, and let $\pi$ be a uniformizer of $F$ which is a root of the Eisenstein polynomial of the extension (see \cite[Equation 2.1]{Pappas-Rapoport}). Fix a positive integer $d$ and a rank-$d$ free module $\Lambda$ for the ring of integers $\cO_{F}$ of $F$. Finally, for each embedding $\phi: F \rightarrow F_0^{\text{sep}}$ of $F$ into a fixed separable closure of $F_{0}$, one fixes an integer $r_\phi$ such that $0 \leq r_{\phi} \leq d$. 

Associated to this data, one defines the \emph{reflex field} $E$ \cite[Equation (2.3)]{Pappas-Rapoport} and its ring of integers $\cO_{E}$. Pappas and Rapoport define a scheme $M$, which they call the \emph{standard model for $GL_{d}$ corresponding to $\br = (r_\phi)_\phi$}. This is a scheme over $\Spec \cO_{E}$.  Additionally, they construct two auxiliary objects $\tilde{M}$ and $N = N(\br)$, which are also schemes over $\cO_{E}$. These three spaces are related by maps:
\begin{align}
  \label{eq:three-spaces}
M \stackrel{\pi}{\longleftarrow} \widetilde{M} \stackrel{\varphi}{\longrightarrow} N
\end{align}
One of Pappas and Rapoport's main theorems is the following.
\begin{Theorem}{\cite[Theorem 4.1]{Pappas-Rapoport}}
 The morphism $\widetilde{M} \stackrel{\varphi}{\longrightarrow} N$ is smooth.
\end{Theorem}

 Let $\kk$ be the residue field of $\cO_E$. Let $V = \Lambda \otimes_{\cO_{F_0}} \kk$ and let $T = \pi \otimes \text{id}_{\kk}$ (these are denoted $W$ and $\Pi$ in \cite[Equation 2.9]{Pappas-Rapoport}). Let $n = \sum r_\phi$. Upon base changing to $\kk$, it is apparent from the definitions that \eqref{eq:three-spaces} becomes:
\begin{align}
  \label{eq:three-spaces-basechanged}
\leftexp{'}{\cS^T} \stackrel{\pi}{\longleftarrow} \tilde{\leftexp{'}{\cS^T}} \stackrel{\varphi}{\longrightarrow} \cN_{n,e}
\end{align}
Here $\leftexp{'}{\cS^T} $ is the subscheme of the Grassmannian $\Gr(n,V)$ consisting of $T$-invariant $n$-planes $U$ such that $\det(\lambda - T\vert_U) = \lambda^n$, $\tilde{\leftexp{'}{\cS^T}}$ is its preimage in the Stiefel variety, and the map $\phi$ comes from a canonical map from $\tilde{\leftexp{'}{\cS^T}}$ to $\Mat_{n \times n}$. See \S \ref{sec:grassmannians-and-associated-schemes} and \S \ref{sec:smoothness-of-the-map-phi} for further explanation.
By base change, we immediately have  the following corollary.
\begin{Corollary}
 The morphism $\tilde{\leftexp{'}{\cS^T}} \stackrel{\varphi}{\longrightarrow} \cN_{n,e}$ is smooth.
\end{Corollary}
\noindent In \S \ref{sec:smoothness-of-the-map-phi} we will explain briefly Pappas and Rapoport's proof of this corollary.

Of central importance to Pappas and Rapoport is whether the space $M$ is flat over $\cO_{E}$. When the $(r_\phi)_\phi$ above differ from each other by more than $1$, they prove that $M$ is not flat for dimension reasons. What remains are the cases where the $(r_\phi)_\phi$ differ from each other by at most one. By \cite[Theorem B (iii)]{Pappas-Rapoport}, flatness in these cases would follow from the reducedness of $\cN_{n,e}$. As a corollary of our work one obtains flatness of $M$ in these cases, and therefore, in the terminology of \cite{Pappas-Rapoport}, that the na\"ive local model coincides with the local model.

\subsection{Our approach}

Our solution follows naturally by building on the ideas of our previous work \cite{Muthiah-Weekes-Yacobi} and earlier work by the first two authors and Kamnitzer \cite{Kamnitzer-Muthiah-Weekes}. One of our stated goals in \cite[\S 1.3.2]{Muthiah-Weekes-Yacobi} is to use the framework developed there to understand the scheme structure of nilpotent orbit closures in positive characteristic. The present work is our first success in this program. 

In \cite{Muthiah-Weekes-Yacobi}, we proved a conjecture of Kreiman, Lakshmibai, Magyar, and Weyman about the equations defining affine Grassmannians for $\SL_n$. Their conjecture is that the defining equations are given by certain explicit \emph{shuffle operators}. 

Our study led us to consider the general situation of a nilpotent operator $T$ on a vector space $V$. We consider the Grassmannian $\Gr(n,V)$, and define a subscheme $\cS^T$ by explicit $T$-\emph{shuffle operators} generalizing the shuffle operators. Our first result (Theorem \ref{thm:primeST-eq-cST}) is that $\leftexp{'}{\cS^T} = \cS^T$, where $\leftexp{'}{\cS^T}$ is the scheme appearing in the work of Pappas and Rapoport as explained above. This provides the link between our previous work and the work of Pappas and Rapoport. 

Theorem \ref{thm:primeST-eq-cST} should be of general interest because $\leftexp{'}{\cS^T}$ is given as an explicit moduli functor, while $\cS^{T}$ is given by explicit homogeneous equations in the Pl\"ucker embedding of the Grassmannian. 

In \cite{Kamnitzer-Muthiah-Weekes}, we developed some techniques for attacking a conjecture about the reduced scheme structure on Schubert varieties in the affine Grassmannian. We use this to prove that a certain family of $\cS^T$ are reduced (Proposition \ref{prop:identifying-first-schuberts}). In particular, we show that these schemes are Schubert varieties in the affine Grassmannian with their reduced scheme structure. Using the fact that Schubert varieties are Frobenius split compatibly with their Schubert subvarieties, we can take intersections and conclude that a larger family of $\cS^T$ are also reduced (Proposition \ref{prop:intersecting-schuberts}).   It follows that all corresponding $\widetilde{\cS^T}$ are reduced, since $\widetilde{\cS^T} \rightarrow \cS^T$ is a $\GL_n$--torsor.

We note that this strategy is similar to the proof of \cite[Theorem 4.5.1]{Gortz-2001}, where the special fiber $\overline{M}^{\text{loc}}$ of a local model is expressed as an intersection of smooth Schubert varieties.  One additional difficulty for us is that we consider intersections of singular Schubert varieties.

Finally, we show that for each $n$ and $e$ that there is some $\cS^T$ (which we have shown is reduced) such that the corresponding map $\widetilde{\cS^T} \rightarrow \cN_{n,e}$ is surjective and smooth. Surjectivity is a direct calculation (Proposition \ref{prop:surjectivity}). Smoothness is due to Pappas and Rapoport (Theorem \ref{thm:pappas-rapoport-smoothness}). They work in a more complicated situation than we consider, so for the benefit of the reader we explain in \S \ref{sec:smoothness-of-the-map-phi} how their proof works in our simpler setting. Therefore, we conclude that $\cN_{n,e}$ is reduced.

After analyzing our argument, we found a simplified but less conceptual version that is almost completely elementary. The only non-elementary ingredient is the Frobenius splitting of Schubert varieties. We present this simplified argument in \S \ref{sec:another-proof}.

\subsection{Acknowledgements}
We would like to thank an anonymous referee of \cite{Muthiah-Weekes-Yacobi} for suggesting the connection between our work and that of Pappas and Rapoport, and Geordie Williamson and Xinwen Zhu for their comments on an earlier version of this manuscript.  We are also grateful to the Centre de Recherches Math\'ematiques and to Joel Kamnitzer and Hugh Thomas for organizing the program ``Quiver Varieties and Representation Theory'' where this project was started. 

D.M. was supported by JSPS KAKENHI Grant Number JP19K14495. A.W. was supported in part by the Perimeter Institute for Theoretical Physics. Research at Perimeter Institute is supported by the Government of Canada through Innovation, Science and Economic Development Canada and by the Province of Ontario through the Ministry of Research, Innovation and Science. O.Y. is supported by the ARC grant DP180102563.

\section{Preliminaries}

\subsection{Grassmannians and associated schemes}
\label{sec:grassmannians-and-associated-schemes}
Let $\kk$ be a field of arbitrary characteristic. Because our goal is to prove reducedness of a scheme, we may assume $\kk$ is algebraically closed. Let $V$ be a finite-dimensional vector space over $\kk$, and let $n$ be a non-negative integer. Write $\Gr(n,V)$ for the Grassmannian of $n$-planes in $V$. Let $\underline{V}$ denote the trivial bundle on $\Gr(n,V)$, and let $\cT$ denote the tautological rank-$n$ vector subbundle of $\underline{V}$. Write $\cO(-1) = \Wedge^n \cT$,  which is a line subbundle of the trivial bundle $\underline{\Wedge^n V}$. This defines the Pl\"ucker embedding $\Gr(n,V) \hookrightarrow \PP(\Wedge^n V)$.

Further, let us assume that $T : V \rightarrow V$ is a nilpotent operator. We denote by $\cG^T(n,V)$ the subscheme of $\Gr(n,V)$ consisting of $T$-invariant subspaces (see e.g.~\cite[\S 5.4]{Muthiah-Weekes-Yacobi}). When $n$ and $V$ are clear from context, we will simply write $\cG^T = \cG^T(n,V)$.

For any $T$-invariant $n$-plane $U$, we can consider the characteristic polynomial $\det(\lambda - T\vert_{U})$. This defines a regular function $\chi: \cG^T \rightarrow \kk[\lambda]$. Notice that set theoretically $\chi$ is equal to $\lambda^n$, but this is not necessarily true scheme-theoretically. In particular, the non-leading coefficients of $\chi$ are all nilpotent elements of $\kk[\cG^T]$. We define $\leftexp{'}{\cS^T}$ to be the closed subscheme of $\cG^T$ defined by the equation $\chi = \lambda^n$.

In \cite[\S 6.1]{Muthiah-Weekes-Yacobi}, we define for each $d \geq 1$ \emph{shuffle operators}, which are linear operators $\sh^T_d : \Wedge^n V \rightarrow \Wedge^n V$. For each $\sh^T_d$ we can consider the vanishing locus, denoted $\VV(\sh_d^T)$, which is the subvariety of $\PP(\Wedge^n V)$ given by the projectivization of the kernel of $\sh_d^T$. By composing with all linear functionals on $\Wedge^n V$, the shuffle operators define a subspace of linear forms on $\PP(\Wedge^n V)$. We will refer to the homogeneous ideal generated by these linear forms as the \emph{shuffle ideal}.

We define $\cS^T = \cS^T(n,V)$ to be the closed subscheme of $\Gr(n,V)$ defined by the shuffle ideal, i.e.~we define $\cS^T = \bigcap_{d \geq 1} \VV(\sh_d^T) \cap \Gr(n,V)$.  In our previous work, we proved the following.
\begin{Theorem}{\cite[Theorem 6.8]{Muthiah-Weekes-Yacobi}}
\label{thm: MWY inclusion}
The subscheme $\cS^T$ is a closed subscheme of $\cG^T$. This closed embedding induces an isomorphism of reduced schemes.
\end{Theorem}

\begin{Remark}
In fact we conjecture (\cite[Conjecture 7.7]{Muthiah-Weekes-Yacobi}) that $\cS^T$ is reduced and is therefore equal to the reduced scheme of $\cG^T$.
\end{Remark}

The \emph{Stiefel variety} $\tilde{\Gr(n,V)}$ is the open subscheme of $\Hom_{\kk}(\kk^n, V)$ consisting of rank--$n$ linear transformations. We have a natural map $\tilde{\Gr(n,V)} \rightarrow \Gr(n,V)$, which is a $\GL_{n}$--torsor. For any closed subscheme $ X \hookrightarrow \Gr(n,V)$, we define $\tilde{X}$ to be the preimage of $X$ inside the Stiefel variety. For example, we define $\tilde{\cS^T}$ this way.

\subsection{Alternative description of $\cS^T$}
\label{sec:alternative-description-of-S-T}

\begin{Theorem}
  \label{thm:primeST-eq-cST}
  We have an equality 
  $ \cS^T = \primeST $
  as subschemes of $\Gr(n,V)$.
\end{Theorem}

\begin{proof}
  
Recall that $\cS^T$ is defined by intersecting the Grassmannian $\Gr(n,V)$ with the vanishing locus of shuffle operators $sh_{d}^T: \Wedge^n V \rightarrow \Wedge^n V$. Consider these  as operators on the trivial bundle $sh_d^T : \underline{\Wedge^n V} \rightarrow \underline{\Wedge^n V}$, which we can restrict to get morphisms:
\begin{align}
  \label{eq:shuffles-as-map-of-bundles}
sh_d^T : \cO(-1) \rightarrow \underline{\Wedge^n V}
\end{align}
One can alternately define $\cS^T$ to be the intersection of the vanishing loci of \eqref{eq:shuffles-as-map-of-bundles}.
Observe by Theorem \ref{thm: MWY inclusion} that over $\cG^T$, the maps $sh_d^T : \cO(-1) \rightarrow \underline{\Wedge^n V}$ factor through maps $sh_d^T: \cO(-1) \rightarrow \cO(-1)$.

Consider the $\kk[z]$-linear operator $I + zT$ on $V \otimes \kk[z]$, and consider its $n$-th wedge power  $\Wedge^n(1+zT)$ (taken over $\kk[z]$), which acts on $\Wedge^n(V) \otimes \kk[z]$.
By \cite[Lemma 6.5]{Muthiah-Weekes-Yacobi}, we have the following equality 
\begin{align}
  \label{eq:wedge-formula-for-shuffle-operators}
  \Wedge^n(I + zT) = I + \sum_{d=1}^n z^d sh_d^T
\end{align}
of operators on $\Wedge^n(V) \otimes \kk[z]$.

Write $\AA^1 = \Spec \kk[z]$. We interpret \eqref{eq:wedge-formula-for-shuffle-operators} as an equality of endomorphisms of $\underline{\Wedge^n V} \otimes \kk[z]$, the trivial bundle on $\Gr(n,V) \times \AA^1$ with fiber $\Wedge^n V$.  If we restrict to $\cG^T \times \AA^1$, we obtain a map:
\begin{align}
  \label{eq:wedge-map-on-GT}
I + \sum_{d=1}^n z^d sh_d^T : \cO(-1)\otimes \kk[z] \rightarrow \cO(-1)\otimes \kk[z]
\end{align}
Because $\cO(-1)$ is a line bundle, the map \eqref{eq:wedge-map-on-GT} is equivalent to the data of a regular function on $\cG^T$ valued in $\kk[z]$. By \eqref{eq:wedge-formula-for-shuffle-operators}, this map is given by the 
$z^n \chi(-z^{-1})$, where $\chi: \cG^T \rightarrow \kk[z]$ is the characteristic polynomial map defined above.
\end{proof}

\subsection{Smoothness of the map $\phi: \tilde{\cS^T} \rightarrow \cN_{n,e}$}
\label{sec:smoothness-of-the-map-phi}

We have a  map $ \phi : \tilde{\cS^T} \rightarrow \Mat_{n \times n}$ defined as follows.
For a test ring $R$, we have $\tilde{\cS^T}(R) = \left\{ (U,\psi) \suchthat U \in \cS^T(R) \textand \psi : R^n \overset{\sim}{\rightarrow} U  \right\}$. The map $\phi$ sends a pair $(U,\psi)$ to the $n \times n$ matrix $\psi^{-1} \circ T \circ \psi$.

Because of the characteristic polynomial equation defining $\cS^T$, the image of $\phi$ lies in the nilpotent cone $\cN$ (even scheme theoretically). In general this map is poorly behaved, but in the special case we consider below it is smooth.

Let $d,e \geq 1$ be integers, let $V$ be a vector space of dimension $de$ and suppose $T$ is the nilpotent endomorphism of $V$ that is given by a standard Jordan form corresponding to the partition $(e^d)$. Let $\cS^T = \cS^T(n,V)$ for some $n \leq de$. In this case, the morphism $\phi$ factors through a map $\phi: \tilde{\cS^T} \rightarrow \cN_{n,e}$.

\begin{Theorem}{\cite[Theorem 4.1]{Pappas-Rapoport}}
  \label{thm:pappas-rapoport-smoothness}
 The morphism $\phi: \tilde{\cS^T} \rightarrow \cN_{n,e}$ is smooth of relative dimension $nd$.
\end{Theorem}

The proof of \cite[Theorem 4.1]{Pappas-Rapoport} holds in a more complicated arithmetic situation, and it simplifies greatly in the present one. For the benefit of the reader, we will briefly explain how their proof works in our situation.

Define a closed subscheme $\cV$ of $\Mat_{n,n} \times \Hom_{\kk}(\kk^n, V)$ as follows. For any test ring $R$, consider the ring $S = R[t]/t^e$. The operator $T$ defines an $S$--module structure on $V\otimes R$. Additionally, any point $A \in \cN_{n,e}(R)$ defines an $S$--module structure on $R^n$, which we denote by  $R^n_{A}$. We define $\cV$ by:
\begin{align}
  \cV(R) = \left\{ (A,\phi) \suchthat A \in \cN_{n,e}(R) \textand \phi \in \Hom_{S}(R^n_A, V \otimes R) \right\}
\end{align}
It is clear that $\tilde{\cS^{T}}$ admits an open embedding into $\cV$ with image equal to the set of pairs $(A,\phi)$ where $\phi$ is of maximal rank.

We also have a map $\cV \rightarrow \cN_{n,e}$, but for general $T$ this map will not be well behaved. However,  our assumptions on the Jordan type of $T$ imply that $V \otimes R$ is a \emph{free} $S$-module of rank $d$. Using this, one concludes that the map $\cV \rightarrow \cN_{n,e}$ is a (trivial) vector bundle of rank $nd$. Therefore, the map $\tilde{\cS^{T}} \rightarrow \cN_{n,e}$ is smooth because it factors as an open embedding into a vector bundle followed by projection from the vector bundle to its base.

\begin{Remark}
For a general $T$ with $T^e = 0$, the fibers of $\cV \rightarrow \cN_{n,e}$ are closely related to the work of Shayman \cite[Section 5]{Shayman}, who studied certain locally closed subvarieties of the varieties $\cS^T$.
\end{Remark}

\subsection{Type $A$ affine Grassmannians}
\label{sec:type-A-affine-Grassmannians}

Let $n \geq 2$. Recall the \emph{$\GL_n$ affine Grassmannian} $\affGr_{\GL_n}$ parameterizing lattices $L \subset \kk((t))^n$ such  $ t^{a} L_0 \subseteq L \subseteq t^{-b} L_0$ for some integers $a,b \geq 0$ where $L_0 = \kk[[t]]^n$ (see e.g. \cite[Section 1.1]{Zhu-2017}).  We further require that $\dim_\kk (L \cap L_0 )/L_0 - \dim_\kk( L \cap L_0) / L = 0$. This condition means that  we are only considering the connected component of the full affine Grassmannian of $\GL_n$ containing the point $L_0$. We will not use the other connected components, so our notation should cause no confusion.

For each pair of integers $a,b \geq 0$, we can consider the linear operator $T_{a,b}$ on $t^{-b} L_0 / t^{a} L_0$ given by multiplication by $t$. We therefore obtain a closed embedding $\cG^{T_{a,b}} = \cG^{T_{a,b}}(na,t^{-b} L_0 / t^{a} L_0) \hookrightarrow \affGr_{\GL_n}$. Explicitly, we identify $\cG^{T_{a,b}}$ with the subscheme of $\affGr_{\GL_n}$ parameterizing lattices $L \subset \kk((t))^n$ such that 
$t^{a}L_0 \subset L \subset t^{-b}L_0$ and $ \dim\big(L / t^{a} L_0\big)=na$.
Taking direct limits we have
\begin{align}
  \label{eq:ind-scheme-presentation-of-Gr-GL-n}
  \affGr_{\GL_n} = \lim_{a,b \rightarrow \infty} \cG^{T_{a,b}}
\end{align}
This realizes $\affGr_{\GL_n}$ as an ind-scheme of ind-finite type.
We define the affine Grassmannian $\affGr_{\SL_n}$ for $\SL_n$ to be the induced reduced structure of $\affGr_{\GL_n}$. Explicitly:
\begin{align}
  \label{eq:ind-scheme-presentation-of-Gr-SL-n}
 \affGr_{\SL_n} =  \lim_{a,b \rightarrow \infty} (\cG^{T_{a,b}})^{red} 
\end{align}

We may also consider $ \cS^{T_{a,b}} = \cS^{T_{a,b}}(na,t^{-b} L_0 / t^{a} L_0) $. A priori, $\cS^{T_{a,b}}$ is a subscheme of $\affGr_{\GL_n}$, but below (Corollary \ref{cor:ST-ab-is-reduced}) we will show that $\cS^{T_{a,b}}$ is in fact equal to the reduced scheme $(\cG^{T_{a,b}})^{red}$ and is therefore a subscheme of $\affGr_{\SL_n}$.

\subsubsection{Big cells}
\label{sec:big-cells}

Consider the decomposition $\kk((t))^n = L_0 \oplus t^{-1}\kk[t^{-1}]^n$ and the induced projection map $\kk((t))^n \rightarrow L_0$. The \emph{big cell} of $\affGr_{\GL_n}$ is the open locus consisting of lattices $L$ such that the restricted projection map $L \rightarrow L_0$ is an isomorphism. As is well known (see e.g.~\cite[Lemma 2]{Faltings}), the big cell is isomorphic to the ind-scheme $\GL_n^{(1)}[t^{-1}]$ whose $R$-points are given by $\{ A(t^{-1}) \in  \GL_n(R[t^{-1}]) \suchthat A(t^{-1}) \equiv 1 \mod t^{-1} \}$. The open immersion $\GL_n^{(1)}[t^{-1}] \hookrightarrow \affGr_{\GL_n}$ is given by sending a matrix polynomial $A(t^{-1})$ to the lattice spanned by its columns. We will use this identification and simply write $\GL_n^{(1)}[t^{-1}]$ for the big cell.

We can intersect the big cell of $\affGr_{\GL_n}$ with $\affGr_{\SL_n}$ to obtain the big cell of the latter. As above, we can identify the big cell of $\affGr_{\SL_n}$ with $\SL_n^{(1)}[t^{-1}]$, which is analogously defined.

Observe that the group $\SL_n[[t]]$ acts on $\affGr_{\GL_n}$ by left multiplication. We note that every orbit meets the big cell as we consider only the $L_0$ connected component of the full affine Grassmannian of $\GL_n$.

\subsubsection{Schubert varieties}
\label{sec:schubert-varieties}

Recall that the $SL_n[[t]]$ orbits on $\affGr_{\SL_n}$ are indexed by dominant cocharacters. For a dominant cocharacter $\mu$, let $\overline{\affGr^{\mu}}$ be the corresponding orbit closure in $\affGr_{\SL_n}$. By definition $\overline{\affGr^{\mu}}$ has reduced scheme structure.  Additionally, $\affGr_{\SL_n}$ is isomorphic to a partial flag variety for the Kac-Moody group $\widehat{\SL_n}$, and the subvarieties $\overline{\affGr^{\mu}}$ are Schubert subvarieties. In particular, each $\overline{\affGr^{\mu}}$ admits a Frobenius splitting compatible with all of its Schubert subvarieties (\cite[Ch. VIII]{Mathieu}, see also \cite[\S 4]{Faltings} and \cite[Chapter 2]{Brion-Kumar}).
Therefore, the scheme theoretic intersection of two Schubert varieties is Frobenius split and hence reduced. This result is the only consequence of Frobenius splitting we will use. 

The Schubert varieties relevant to us are $\overline{ \affGr^{p n \varpi_1}}$ and $\overline{ \affGr^{q n \varpi_{n-1}}}$, where $p,q \geq 1$ are integers. Here $n \varpi_1$ and $n \varpi_{n-1}$ are $n$ times the first and last fundamental coweight, which are the minimal multiples that lie in the cocharacter (equivalently, coroot) lattice for $\SL_n$.
We have
\begin{align}
  \label{eq:1}
\overline{\affGr^{p n \varpi_1}} = (\cG^{T_{p(n-1),p}})^{red},   \qquad \overline{\affGr^{q n \varpi_{n-1}}} = (\cG^{T_{q,q(n-1)}})^{red}
\end{align}
which is an explicit set-theoretic description of these Schubert varieties in terms of lattices.
Because of Theorem \ref{thm: MWY inclusion}, one immediately obtains closed embeddings
\begin{equation}
\label{eq: embedding Schubert varieties into S^T}
  \overline{\affGr^{p n \varpi_1}} \hookrightarrow \cS^{T_{p(n-1),p}}, \qquad  \overline{\affGr^{q n \varpi_{n-1}}} \hookrightarrow \cS^{T_{q,q(n-1)}}
\end{equation}
that are bijective on points. Below (Proposition \ref{prop:identifying-first-schuberts}) we will show that these are actually isomorphisms of schemes. In particular, the na\"ive lattice description of these Schubert varieties is scheme-theoretically correct once we add a further condition about the characteristic polynomial of $t$ (as in the defintion of $\primeST$).

\subsubsection{Diagram automorphism}
\label{sec:diagram-automorphism}

The presentation \eqref{eq:ind-scheme-presentation-of-Gr-GL-n} gives rise to a closed embedding of $\affGr_{\GL_n}$ into the \emph{Sato Grassmannian} $\SGr = \lim_{a,b \rightarrow \infty} \Gr(na,t^{-b} L_0 / t^{a} L_0)$. Taking direct limits of the Pl\"ucker embeddings gives the Pl\"ucker embedding of $\SGr$ into $\PP(\cF)$, where $\cF$ is the Fermion Fock space. As is well known (see for example \cite[\S 4.1]{Muthiah-Weekes-Yacobi}), the Fermion Fock space can be identified with the vector space underlying the ring of symmetric functions.

Consider the involution of the ring of symmetric functions given by swapping elementary and homogeneous symmetric functions (denoted $\omega$ in \cite[Section I.2]{Macdonald}). This induces an involution $\omega$ of $\SGr$ that preserves $\affGr_{\GL_n}$ and $\affGr_{\SL_n}$.  On the level of Schubert varieties, $\omega$ maps $\overline{\affGr^\mu}$ isomorphically onto $\overline{\affGr^{\mu^*}}$ where $*$ is induced by the unique non-trivial diagram automorphism of $\SL_n$. In particular, $\omega$ maps $\overline{\affGr^{p n \varpi_1}}$ isomorphically onto $\overline{\affGr^{p n \varpi_{n-1}}}$. On the level of big cells $\GL_n^{(1)}[t^{-1}]$  and $\SL_n^{(1)}[t^{-1}]$, $\omega$ is given by the following map on matrix polynomials:
\begin{equation}
\label{eq: involution on big cell of GL_n}
A(t^{-1} ) \ \longmapsto \ J \cdot A \big( (-1)^n t^{-1} \big)^{-T} \cdot J^{-1}
\end{equation}
where $J $ is the $n\times n$ antidiagonal matrix with entry $(-1)^j$ in column $j$, and where the superscript $-T$ denotes the inverse transpose.

\section{The main proof}
\label{sec:the-main-proof}

\begin{Proposition}
  \label{prop:identifying-first-schuberts}
  For all $p,q \geq 1$, the embeddings $\overline{\affGr^{p n \varpi_1}} \hookrightarrow \cS^{T_{p(n-1),p}}$ and $\overline{\affGr^{q n \varpi_{n-1}}} \hookrightarrow \cS^{T_{q,q(n-1)}}$ from \eqref{eq: embedding Schubert varieties into S^T} are isomorphisms. In particular, $\cS^{T_{p(n-1),p}}$ and $\cS^{T_{q,q(n-1)}}$ are reduced.
\end{Proposition}

\begin{proof}

  First consider the embedding $\overline{\affGr^{p n \varpi_1}} \hookrightarrow \cS^{T_{p(n-1),p}}$.
 Both spaces are invariant under the action of $\SL_n[[t]]$, so it suffices to check that this map is an isomorphism on the big cell $\GL_n^{(1)}[t^{-1}] \subseteq \affGr_{GL_n}$.  Observe that $\cS^{T_{p(n-1),p}} \cap \GL_n^{(1)}[t^{-1}]$ maps into the closed subfunctor of $\GL_n^{(1)}[t^{-1}]$ consisting of matrix polynomials with degree less than or equal to $p$. In \cite[Corollary 2]{Kamnitzer-Muthiah-Weekes}  (see also the proof of \cite[Proposition 6.1]{Beauville-Laszlo}), we proved that $\overline{\affGr^{p n \varpi_1}} \cap \GL_n^{(1)}[t^{-1}]$ is equal as a scheme to the subscheme of matrix polynomials $A(t^{-1}) = 1 + A_1 t^{-1} + \cdots + A_p t^p$ such that $\det(A(t^{-1})) = 1$.  In \cite[Theorem 4.27]{Muthiah-Weekes-Yacobi} we showed that the coefficients of $\det(A(t^{-1}))$ lie in the shuffle ideal  defining $\cS^{T_{p(n-1),p}}$. Therefore the embedding is an isomorphism.

The second isomorphism follows by applying the diagram automorphism (\S \ref{sec:diagram-automorphism}) and observing that the shuffle ideal considered in \cite{Muthiah-Weekes-Yacobi} is invariant under it (see \cite[Corollary 4.13, Theorem 4.27]{Muthiah-Weekes-Yacobi}). %

\end{proof}

\begin{Remark}
  Let $L$ be an $R$-point of $\cS^{T_{p(n-1),p}} \cap \GL_n^{(1)}[t^{-1}]$. Then $L$ may be uniquely represented by a matrix polynomial $A(t^{-1}) \in \GL_n^{(1)}[t^{-1}]$. It is not difficult to show that there is an $R$-basis of $t^{-p} L_0 / L$ under which the matrix of $t$ is the companion matrix of the matrix polynomial $A(t^{-1}) = 1 + A_1 t^{-1} + \cdots + A_p t^{-p}$. In particular, the characteristic polynomial of this matrix is $\lambda^{pn} \cdot A(\lambda^{-1})$. We see that the characteristic polynomial equation defining $\cS^{T_{p(n-1),p}}$ (via Theorem \ref{thm:primeST-eq-cST}) corresponds exactly to the equation $\det(A(t^{-1})) = 1$ defining $\overline{\affGr^{p n \varpi_1}} \cap \GL_n^{(1)}[t^{-1}]$ inside matrix polynomials with degree less than or equal to $p$. This in particular gives another proof of \cite[Theorem 4.27]{Muthiah-Weekes-Yacobi}.
\end{Remark}

\begin{Proposition}
  \label{prop:intersecting-schuberts}
  Let $p,q \geq 1$, and let $M = \min \{ q, p(n-1) \}$ and $N = \min \{ p, q(n-1) \}$.
 Then $\cS^{T_{M,N}}$ is reduced.
\end{Proposition}

\begin{proof}
  It is clear that $\cS^{T_{M,N}}$ is equal set theoretically to the intersection $\overline{\Gr^{pn\varpi_1^\vee}} \cap \overline{\Gr^{qn\varpi_{n-1}^\vee}}$. Furthermore, it is known that the scheme-theoretic intersection $\overline{\Gr^{pn\varpi_1^\vee}} \cap \overline{\Gr^{qn\varpi_{n-1}^\vee}}$ is reduced by Frobenius splitting (see \S \ref{sec:schubert-varieties}).

Denote $K = \max\{p, q(n-1)\}$.  Then we have a Cartesian diagram of closed embeddings
\begin{equation}
\begin{tikzcd}
  \overline{\Gr^{pn\varpi_1^\vee}} \cap \overline{\Gr^{qn\varpi_{n-1}^\vee}} \arrow[d, hook] \arrow[r, hook] & \overline{\Gr^{pn\varpi_1^\vee}} \arrow[d, hook] \\
  \overline{\Gr^{qn\varpi_{n-1}^\vee}} \arrow[r, hook] & \overline{\Gr^{K n\varpi_1^\vee}}
\end{tikzcd}
\end{equation}
and a diagram of closed embeddings
\begin{equation}
\begin{tikzcd}
  \cS^{T_{M,N}} \arrow[d, hook] \arrow[r, hook] & \cS^{T_{p(n-1),p}} \arrow[d, hook] \\
  \cS^{T_{q,q(n-1)}} \arrow[r, hook] & \cS^{T_{K(n-1),K}} 
\end{tikzcd}
\end{equation}
which is not a priori Cartesian.
By Proposition \ref{prop:identifying-first-schuberts}, the objects in these two diagrams agree in all but the northwest corner. By the universal property of the fiber product, we obtain a closed embedding $\cS^{T_{M,N}} \hookrightarrow \overline{\Gr^{pn\varpi_1^\vee}} \cap \overline{\Gr^{qn\varpi_{n-1}^\vee}} $. Because $\overline{\Gr^{pn\varpi_1^\vee}} \cap \overline{\Gr^{qn\varpi_{n-1}^\vee}}$ is reduced and this map is a bijection on points, it is an isomorphism.
\end{proof}

\begin{Remark}
\label{rmk:intersecting-schuberts}
The intersection $\overline{\affGr^{p n \varpi_1}} \cap \overline{\affGr^{q n \varpi_{n-1}}}$ is irreducible and is therefore also a Schubert variety for $\affGr_{\SL_n}$ (see \cite[Proposition 5.4]{Kamnitzer-Muthiah-Weekes}).  In fact it is also not hard to see that $\cS^{T_{q,p}}$ is isomorphic to $\overline{\affGr^{p n \varpi_1}} \cap \overline{\affGr^{q n \varpi_{n-1}}}$ and is in particular reduced.
\end{Remark}

\begin{Proposition}
  \label{prop:surjectivity}
  Let $e$ be an integer with $n \geq e \geq 1$.
The map $\phi : \widetilde{\cS^{T_{1,e-1}}} \rightarrow \cN_{n,e}$ is surjective. 
\end{Proposition}

\begin{proof}
As surjectivity is a set-theoretic statement, it suffices to work with reduced schemes.
Dividing with remainder, write $n = c e + f$ where $0 \leq f < e$, and let $\tau = (e^c,f)$ be the partition with $c$ parts of size $e$ and one part of size $f$. Then the reduced scheme of $\cN_{n,e}$ is exactly the nilpotent orbit closure $\overline{\OO_\tau} \subseteq \Mat_{n\times n}$. The map $\phi$ is $\GL_n$-equivariant, so it suffices to check that every nilpotent orbit $\OO_\sigma \subset \overline{\OO_\tau}$ has non-empty intersection with the image. This is a straightforward calculation.
\end{proof}

\begin{Theorem}
  \label{thm:main-theorem}
 For every $n$, $e$ with $1 \leq e \leq n$, the scheme $\cN_{n,e}$ is reduced. 
\end{Theorem}

\begin{proof}
The scheme $\cS^{T_{1,e-1}}$ is reduced by Proposition \ref{prop:intersecting-schuberts} (choosing $p=e-1$ and $q=1$).  The map $\widetilde{\cS^{T_{1,e-1}}} \rightarrow \cS^{T_{1,e-1}}$ is a torsor for $\GL_n$, so $\widetilde{\cS^{T_{1,e-1}}}$ is also reduced.  By Theorem \ref{thm:pappas-rapoport-smoothness} and Proposition \ref{prop:surjectivity}, the map $\phi : \widetilde{\cS^{T_{1,e-1}}} \rightarrow \cN_{n,e}$ is smooth and surjective. In particular, it is faithfully flat, and the property of being reduced descends along faithfully flat morphisms.
\end{proof}

\subsection{Another proof}
\label{sec:another-proof}

After closely inspecting the above argument, we found the following direct proof of Theorem \ref{thm:main-theorem}. 
The idea is essentially the same as the more conceptual proof above, but gets to the punchline quickly and by more elementary means. The only non-elementary ingredient is the Frobenius splitting of Schubert varieties.

\begin{proof}[Another proof of Theorem \ref{thm:main-theorem}]

  For any $p \geq 0$, define $Z_p$ to be the closed subscheme of $\Mat_{n \times n}[t^{-1}]$ that represents the functor sending any test ring $R$ to:
  \begin{align*}
Z_{p}(R) = \left\{ 1 + B_1 t^{-1} +  \cdots + B_{p} t^{-p} \in \Mat_{n \times n}(R[t^{-1}]) \suchthat \det(1 + B_1 t^{-1} \cdots + B_{p} t^{p)}) = 1 \right\}
  \end{align*}
  We see that $Z_p$ is a subscheme of $\SL_n^{(1)}[t^{-1}]$, and it is equal as a scheme to the intersection $\overline{\Gr^{pn \varpi_1 }} \cap \SL_n^{(1)}[t^{-1}]$ by \cite[Corollary 2]{Kamnitzer-Muthiah-Weekes} (see also the proof of \cite[Proposition 6.1]{Beauville-Laszlo}). In particular, $Z_p$ is reduced.

Let $X$ be the closed subscheme $\Mat_{n \times n}[t^{-1}]$ that represents the functor:
  \begin{align*}
X(R) = \left\{ 1 + C_1 t^{-1} + \cdots + C_{n-1} t^{-(n-1)}  \in \Mat_{n \times n}(R[t^{-1}]) \suchthat \det(1 - C_1 t^{-1})   = 1 \textand   C_i = C_1^i \right\}
  \end{align*}
  Observe that the condition $\det(1 - C_1 t^{-1})=1$ is equivalent to requiring that the characteristic polynomial of $C_1$ is $\lambda^n$. By the Cayley-Hamilton theorem, we have $C_1^n = 0$. Therefore,  
$$
1 + C_1 t^{-1} + \cdots + C_{n-1} t^{-(n-1)} = (1 - C_1t^{-1})^{-1}
$$
and $\det(1 + C_1 t^{-1} + \cdots + C_{n-1} t^{-(n-1)}) = 1$. So $X$ is equal to a closed subscheme of $\overline{\Gr^{(n-1)n \varpi_1 }} \cap \SL_n^{(1)}[t^{-1}]$.

The involution \eqref{eq: involution on big cell of GL_n} of $\SL_n^{(1)}[t^{-1}]$ gives an isomorphism between $X$ and $Z_1$.  We therefore conclude that $X$ is equal as a scheme to the intersection $\overline{\Gr^{n \varpi_{n-1} }} \cap \SL_n^{(1)}[t^{-1}]$. In particular, $X$ is reduced. More directly, the map $X \rightarrow \Mat_{n \times n}$ sending 
$1 + C_1 t^{-1} + \cdots + C_{n-1} t^{-(n-1)}$ to $C_1$ induces an isomorphism between $X$ and the nilpotent cone $\cN$. 

Now consider the scheme theoretic intersection $X \cap Z_{e-1}$, which we compute inside $Z_{n-1}$. Because $Z_{n-1}$ admits a Frobenius splitting compatible with $X$ and $Z_{e-1}$, the intersection $X \cap Z_{e-1}$ is reduced.

On the other hand, we directly compute $(X \cap Z_{e-1})(R)$ to be the set of $1 + C_1 t^{-1} + \cdots + C_{n-1} t^{-(n-1)} \in \Mat_{n \times n}(R[t^{-1}])$ such that $\det(1 - C_1 t^{-1})   = 1$ , $C_i = C_1^i$, and $C_i = 0$ for $i \geq e$.  Therefore, under the isomorphism $X \rightarrow \cN$, we see that the $X \cap Z_{e-1}$ maps isomorphically to $\cN_{n,e}$.
\end{proof}

Recall that all schemes of the form $\cS^{T_{a,b}}$ are reduced (Proposition \ref{prop:intersecting-schuberts}, Remark \ref{rmk:intersecting-schuberts}). We can also argue this directly using the proof above.  For all $a,b$, the map $\phi : \widetilde{\cS^{T_{a,b}}} \rightarrow \cN_{na,a+b}$ is smooth (but not necessarily surjective).  Because reducedness ascends along smooth maps, and $\widetilde{\cS^{T_{a,b}}}\rightarrow \cS^{T_{a,b}}$ is a torsor for $\GL_{na}$, we obtain another proof of the following: 
\begin{Corollary}
  \label{cor:ST-ab-is-reduced}
  For all $a,b \geq 1$, the scheme $\cS^{T_{a,b}}$ is reduced.
\end{Corollary}

\bibliographystyle{amsalpha}
\bibliography{on_a_conjecture_of_pappas_and_rapoport}

\end{document}